\documentclass[11pt]{article}

\usepackage{amsmath} 
\usepackage{amssymb}
\usepackage{amsfonts}
\usepackage{latexsym}
\usepackage{amsthm}
\usepackage{amsxtra}
\usepackage{amscd}
\usepackage{bbm}
\usepackage{mathrsfs}
\usepackage{bm}
\usepackage{xcolor}
\usepackage{subcaption}
\usepackage[utf8]{inputenc}

\usepackage{graphicx, color}

\definecolor{darkred}{rgb}{0.9,0,0.3}
\definecolor{darkblue}{rgb}{0,0.3,0.9}

\usepackage{hyperref}

\theoremstyle{plain} 
\newtheorem{theorem}{Theorem}[section]
\newtheorem*{theorem*}{Theorem}
\newtheorem{lemma}[theorem]{Lemma}
\newtheorem{assumption}[theorem]{Assumption}
\newtheorem*{lemma*}{Lemma}

\newtheorem*{corollary*}{Corollary}
\newtheorem{proposition}[theorem]{Proposition}
\newtheorem*{proposition*}{Proposition}
\newtheorem{definition}[theorem]{Definition}
\newtheorem*{definition*}{Definition}

\newtheorem*{conjecture*}{Conjecture}

\theoremstyle{definition} 

\newtheorem*{example*}{Example}
\newtheorem{remark}[theorem]{Remark}
\newtheorem*{remark*}{Remark}

\renewcommand{\P}{\mathbb{P}}

\renewcommand{\H}{\mathbb{H}}

\renewcommand{\leq}{\leqslant}
\renewcommand{\geq}{\geqslant}
\renewcommand{\epsilon}{\varepsilon}

\def\brb#1{\left[#1\right]}

\title{Continuity in $\kappa$ in $SLE_\kappa$ theory using a constructive method and Rough Path Theory}
\author{Dmitry Beliaev, Terry J. Lyons,  Vlad Margarint}

\begin{document}

\maketitle

\begin{abstract}
Questions regarding the continuity in  $\kappa$ of the $SLE_{\kappa}$ traces and maps appear very naturally in the study of SLE.
In order to study the first question, we consider a natural coupling of SLE traces: for different values of $\kappa$ we use the same Brownian motion. It is very natural to assume that with probability one, $SLE_\kappa$ depends continuously on $\kappa$. It is rather easy to show that $SLE$ is continuous in the Carath\'eodory sense, but showing that $SLE$ traces are continuous in the uniform sense is much harder. In this note we show that for a given sequence $\kappa_j\to\kappa \in (0, 8/3)$, for almost every Brownian motion $SLE_\kappa$ traces converge locally uniformly. This result was also recently obtained by Friz, Tran and Yuan using different methods. In our analysis, we provide a constructive way to study the $SLE_{\kappa}$ traces for varying parameter $\kappa \in (0, 8/3)$. 
The argument is based on a new dynamical view on the approximation of SLE curves by curves driven by a piecewise square root approximation of the Brownian motion.
 
The second question can be answered naturally in the framework of Rough Path Theory. Using this theory, we prove that the solutions of the backward Loewner Differential Equation driven by $\sqrt{\kappa}B_t$ when started away from the origin are continuous in the $p$-variation topology in the parameter $\kappa$, for all $\kappa \in \mathbb{R}_+$.
\end{abstract}

\section{Introduction}
The Schramm-Loewner evolution $SLE_{\kappa}$ is a one-parameter family of random planar growth processes constructed as a solution to Loewner equation when the driving term is a Brownian motion with diffusivity $\kappa>0$. It was introduced in  \cite{schramm2000scaling} by Oded Schramm, in order to give meaning to the scaling limits of Loop-Erased Random Walk and Uniform Spanning Trees.

The problem of continuity of the traces generated by Lowener chains was studied in the context of chains driven by bounded variation drivers in \cite{shekhar2019remarks}, where the continuity of the traces generated by the Loewner chains was established. Also, the question appeared in \cite{lind2010collisions}, where the Loewner chains were driven by H\" older-$1/2$ functions with norm bounded by $\sigma$ with $\sigma < 4\,.$ In this context, the continuity of the corresponding traces was established with respect to the uniform topologies on the space of drivers and with respect to the same topology on the space of simple curves in $\H$. Another paper that addressed a similar problem is \cite{sheffield2012strong}, in which the condition $\|U\|_{1/2}<4$ is avoided at the cost of assuming some conditions on the limiting trace. Some stronger continuity results are obtained in  \cite{friz2017existence} under the assumption that the driver has finite energy, in the sense that $\dot{U}$ is square integrable. 

The question appeares naturally when considering the solution of the corresponding welding problem in \cite{astala2011random}. In this paper it is proved that the trace obtained when solving the corresponding welding problem is continuous in a parameter that appears naturally in the setting. In the context of $SLE_{\kappa}$ traces the problem was studied in \cite{viklund2014continuity}, where the continuity in $\kappa$ of the $SLE_{\kappa}$ traces was proved for any $\kappa <2.1$. In \cite{friz2019regularity} the authors proved the continuity of the traces for $\kappa <8/3.$ 

 We emphasize that our proof uses a result from \cite{friz2019regularity} but the mehod of showing the continuity is different  from the one presented in \cite{friz2019regularity}. In particular, our method gives a constructive way to prove the continuity in $\kappa$ for $\kappa \in (0, 8/3)$ by square root interpolating the Brownian motion driver. 

Another element of the analysis is Rough Path Theory \cite{lyons1998differential} introduced by Terry Lyons in 1998. The theory provides a deterministic platform to study stochastic differential equations which extends both Young's integration and stochastic integration theory beyond regular functions and semi-martingales. Also, Rough Path Theory provides a method of constructing solutions to differential equations driven by paths that are not of bounded variation but have controlled roughness. 
In this note, we use Rough Path Theory in order to study the backward Loewner differential equation started away from the origin. More precisely, we first show that the backward Loewner differential equation driven by $\sqrt{\kappa}B_t$ started away from singularity is a Rough Differential Equation as in Rough Path Theory and then we prove the continuity of the solutions of this equation in the parameter $\kappa$ in the Rough Path $p$-variation topology.

\textbf{Acknowledgement:} T.L and V.M were  supported by ERC Advanced Grant (Grant Agreement No.291244 Esig), V.M. was funded by EPSRC grant 1657722, D.B. and V.M. were partially funded by EPSRC Fellowship  EP/M002896/1. V.M. acknowledges also the support of NYU-ECNU Institute of Mathematical Sciences at NYU Shanghai. We thank Huy Tran and Yizheng Yuan for reading the draft and offering useful suggestions.

\section{Preliminaries and the main results}
The \emph{forward (chordal) Loewner evolution} driven by function $\lambda(t)$, $t\in[0,T]$ is defined as the solution of the following ODE
\begin{equation}
\label{eq:forward LE}
\partial_t g_t(z)=\frac{2}{g_t(z)-\lambda(t)}
,\quad g_0(z)=z, \ z\in\H.
\end{equation}
The corresponding \emph{backward Loewner evolution} is the solution of
\begin{equation}
\label{eq: backward LE}
\partial_t h_t(z)=-\frac{2}{h_t(z)-\lambda(t)}
,\quad h_0(z)=z, \ z \in\H.
\end{equation}
The connection between them is that if we take the driving function $\lambda(T-t)$ in the backward evolution, then $h_T=g_T^{-1}$.

It is a standard fact that $g_t : \H_t=\H\setminus K_t \to \H$  where $\H_t$ is the set of points where the solution exists up to time $t$. Under certain assumptions there is a continuous curve $\gamma(t)=\lim_{y\to 0}g_t^{-1}(\lambda(t)+iy)$. This curve is called the \emph{trace} of the Loewner evolution and $\H_t$ is the unbounded component of $\H\setminus \gamma([0,t])$.

Stochastic Loewner Evolution $SLE_\kappa$ with $\kappa\ge 0$ is the Loewner Evolution driven by $\sqrt{\kappa}B_t$, where $B_t$ is the standard Brownian motion. It is known that the trace of $SLE$ exist almost surely. 

It is natural to ask how the trace of SLE depends on $\kappa$. This question has many different interpretations. In this note we are interested in the following form. Given a sample of the Brownian motion, we would like to show that with probability one, $SLE_\kappa$ trace $\gamma^\kappa(t)$ is continuous in $\kappa$ in the metric of the sup-norm on $[0,T]$ for any $T$.

Johansson Viklund, Rohde and Wong proved \cite{viklund2014continuity} this type of continuity for $\kappa \in [0, 8(2-\sqrt{3}))$. Friz, Tran and Yuan \cite{friz2019regularity} obtained a similar result for $\kappa<8/3$. In this note we obtain a constructive method to compare $SLE_{\kappa}$ traces with varying parameter $\kappa$ and prove their continuity in $\kappa$ for $\kappa \in (0, 8/3)$. Compared with the method in \cite{friz2019regularity}, we show how one can compare the $SLE_{\kappa}$ traces using an approximation algorithm developed in \cite{tran2015convergence}. We emphasize that the algorithm in \cite{tran2015convergence} is defined for fixed $\kappa$,  while in this note, we construct a dynamical version of this algorithm to compare $SLE_{\kappa}$ traces for  $\kappa \in (0, 8/3)$. 

\begin{theorem}\label{first main result}
Let us fix $\kappa \in (0, 8/3)$ and a sequence $\kappa_j\to \kappa$, then for every $T$ and almost every $\omega$
\[
\sup_{t\in[0,T]} |\gamma^{\kappa_j}(t)-\gamma^{\kappa}(t)|\to 0,
\]
where $\omega$ is an element of the probability space on which we define the Brownian motion. Moreover, for the curves  $\gamma^{n, \kappa_j}(t)$ generated by the square-root interpolation of the drivers $\sqrt{\kappa_j}B_t$, we have that for almost every Brownian motion 
\[
\sup_{t \in [0, T]}|\gamma^{n, \kappa_j}(t)-\gamma^{n, \kappa}(t)| \to 0.
\]
\end{theorem}

\section{Proof of the first main result}

\subsection{Deterministic results}

We will use a rather natural approximation algorithm that was introduced by Marshall and Rohde and subsequently used by many authors. The ODE structure of the Loewner evolution implies that $g_{t+s}$ can be written as a composition of two Loewner evolutions, one run by $\lambda$ on $[0,t]$ and the other by $\lambda$ on $[t,t+s]$. The same is true if we split the initial time interval into many small time interval. This suggests that if we can split the time interval (we will always use $[0,1]$ for the sake of simplicity) into small intervals and on each interval we approximate the driver by a simple function for which the Loewner evolution can be solved explicitly, then we can approximate the original Lowener evolution by a composition of many relatively simple explicit functions.

There are, essentially, only two cases where the Loewner evolution can be solved explicitly: when the driver is constant and when it is a multiple of the square root. If the driver function is H\"older-$1/2$ continuous or weakly H\"older-$1/2$ continuous, then  it is natural to approximate the driver by a piecewise square root function. To be more precise, we fix some integer $n$ and consider $t_k=k/n$, $k=0,\dots,n$. For a driving function $\lambda(t)$ we define the approximation $\lambda^n(t)$ which is defined by
\begin{equation}
\label{eq: lambda n def}
\lambda^n(t)=\sqrt{n}(\lambda(t_{k+1})-\lambda(t_k))\sqrt{t-t_k}+\lambda(t_k)\quad t\in [t_k, t_{k+1}].
\end{equation}
This is a piecewise square root function which coincides with $\lambda$ at all $t_j$. 

It is known \cite{kager2004exact} that if the driving function is of the form $\lambda(t)=c\sqrt{t}+d$
then the Loewner evolution can be solved explicitly and $g_t$ is a relatively simple Christoffel-Schwarz type function whose explicit form is not that important. What is important, is that the corresponding trace is a straight interval. Its length is proportional to $\sqrt{t}$ with constant which explicitly depends on $c$ and it makes an angle $\alpha\pi$ with the positive real axis where 
\[
\alpha=\frac{1}{2}-\frac{1}{2}\frac{c}{\sqrt{16+c^2}}.
\]

From now on we make some assumptions about the regularity of the driving function. 

\begin{assumption}
\label{assumption: osc}
A function $\lambda$ is  weakly H\"older-$1/2$ continuous. This means that there exists a subpower function $\phi$  (that is the function growing at infinity slower that any positive power) such that for all $\delta>0$
\begin{equation}
\label{eq: osc definition}
osc(\lambda, \delta):=\sup\{|\lambda(t)-\lambda(s)|: s, t \in [0,1], |t-s|\leq \delta\} \leq \sqrt{\delta}\phi\left(\frac{1}{\delta}\right) .
\end{equation}
\end{assumption}

\begin{assumption}
\label{assumption: derivative control}
There exist $c_0 > 0$,  $y_0 > 0$ and $0 < \beta < 1$ such that 
\[
|\hat{f}'_t(iy)|\le c_0y^{-\beta}, \quad \forall y\le y_0,
\]
where $\hat{f}_t(z)=g_t^{-1}(z+\lambda(t))$.
\end{assumption}

It is known \cite{viklund2011optimal} that if the Loewner evolution satisfies Assumptions \ref{assumption: osc} and \ref{assumption: derivative control}, then there is a trace.  Under the same assumptions Tran proved that LE trace generated by the driving function $\lambda^n$ converges to the trace generated by $\lambda$. 

\begin{theorem}[Theorem $2.2.$ of \cite{tran2015convergence}]\label{Ttran}
Let us assume that the driving function $\lambda(t)$ satisfies Assumptions \ref{assumption: osc} and \ref{assumption: derivative control}. Let $\lambda^n$ be a square root approximation defined by \eqref{eq: lambda n def}. Let $\gamma$ and $\gamma^n$ be the corresponding traces.  Then there exists  a sub-power function $\tilde{\phi}(n)$ which depends on $\phi$, $c_0$ and $\beta$ (from Assumptions mentioned above), such that for all $n \geq \frac{1}{y_0^2}$ and $t \in [0,1]$ we have that
\[
|\gamma^n(t)-\gamma(t) | \leq \frac{\tilde{\phi}(n)}{n^{\frac{1}{2}\left(1-\sqrt{\frac{1+\beta}{2}}\right)}}.
\]
\end{theorem}

This theorem shows that $\gamma^n$ converges uniformly to $\gamma$, moreover, we have a control of the rate of convergence in terms of $\beta$. 

Theorem \ref{Ttran} is one of the main ingredients in our proof. Beyond it, we will also need several technical results that were proved before. We reproduce them here for readers' convenience.

First, following \cite{tran2015convergence}, we define 
\begin{equation}
\label{eq: def of A}
A_{n,c, \phi}= \bigg\{ x+iy \in \mathbb{H}: |x| \leq \frac{\phi(n)}{\sqrt{n}}, \frac{1}{\sqrt{n}\phi(n)} \leq y \leq \frac{c}{\sqrt{n}} \bigg\}\,.
\end{equation}

To shorten many formulas we will use the following notations. Recall that $t_k=k/n$. We define $\gamma_k$ to be the image of $\gamma$ under $g_{t_k}-\lambda$, namely,
\[
\gamma_k(s)=g_{t_k}(\gamma(t_k+s))-\lambda(t_k), \quad 0\le s\le 1-t_k.
\]   
In the same way we define
\[
\gamma_k^n(s)=g^n_{t_k}(\gamma^n(t_k+s))-\lambda^n(t_k), \quad 0\le s\le 1-t_k.
\]

We would like to say that  $\gamma_k(1/n)$ is in some $A_{n,c,\psi}$, but unfortunately this might be false. Instead we have the following Lemma. 
 
\begin{lemma}[Lemma $3.2$ of \cite{tran2015convergence}]\label{SLEpointlemma}
There exists a subpower function $\psi$ depending only on $\phi, c_0$ and $\beta$ (as in Assumptions \ref{assumption: osc} and \ref{assumption: derivative control}) such that for $n \geq 1$ and $0 \leq k \leq n-1$, there exists $s \in [0, \frac{2}{n}]$ such that $\gamma_k(s) \in A_{n, 2\sqrt{2}, \psi}$.
\end{lemma}

For $\gamma_k^n$ we have a similar, but slightly simpler, estimate.

\begin{lemma}[Lemma $3.3$ of \cite{tran2015convergence}]\label{LTran2}
There exists a subpower function $\tilde{\psi}$ depending only on $\phi, c_0$ and $\beta$ (as in Assumptions \ref{assumption: osc} and \ref{assumption: derivative control}) such that $\gamma_k^n(r)$ is in the box $A_{n, 2\sqrt{2}, \tilde{\psi}}$ for $n \geq 1,$ $0 \leq k \leq n-1$ and $r \in [\frac{1}{n}, \frac{2}{n}]$.
\end{lemma}

We will need the following result describing the uniform continuity of traces.

\begin{lemma}[Proposition $3.8$ of \cite{viklund2011optimal}]\label{unifcont}
Let us consider a Loewner evolution satisfying Assumptions \ref{assumption: osc} and \ref{assumption: derivative control}. Then, there exists a subpower function $\phi_1$ such that if $0 \leq t \leq t+s \leq 1$, we have that
\begin{align}
|\gamma(t+s)-\gamma(t)|\leq \phi_1\left(\frac{1}{y}\right)\frac{2}{1-\beta}y^{1-\beta}
\end{align} 
for $0 \leq s \leq y^2 \leq y_0^2$.
\end{lemma}

Finally, we will need a result stating that Loewner evolutions with close drivers are close to each other away from the real line. 

\begin{lemma}[Lemma $2.3$ of \cite{viklund2014continuity}]
\label{l: h uniform close}
Let $0<T<\infty$. Suppose that for $t \in [0,T]$, $h_t^{(1)}$ and $h_t^{(2)}$ satisfy the backward Loewner differential equation \eqref{eq: backward LE} with drivers $\lambda_t^{(1)}$ and $\lambda_t^{(2)}$. Let 
\[
\epsilon=\sup_{s \in [0,T]}|\lambda_s^{(1)}-\lambda_s^{(2)}|.
\]
Then for $u=x+iy\in\H$ we have
\begin{multline*}
|h^{(1)}_T(u)-h^{(2)}_T(u)| 
\\
\leq \epsilon \exp \left[\frac{1}{2}\left( \log \frac{I_{T, y}|(h^{(1)}_{T})^{'}(u)|}{y}\log\frac{I_{T, y}|(h^{(2)}_{T})^{'}(u)|}{y} \right)^{1/2} +\log\log \frac{I_{T, y}}{y}\right],
\end{multline*}
where $I_{T,y}=\sqrt{4T+y^2}$.
\end{lemma}

We will need a slight modification of this result, i.e. we need to apply it for two different staring points $u_1$ and $u_2$ with the same imaginary value $y=\text{Im}u_1=\text{Im}u_2$. Let $z_s^{(j)}:=h_s^{(j)}(z)-\lambda^{(j)}_s$, for $j=1,2$. Following, the proof of the result in \cite{viklund2014continuity}, one can integrate the differential equation for $H(s)=h_s^{(1)}(z)-h_s^{(2)}(z)$, i.e. 
$$\dot{H}(s)-H(s)\psi(s)=(\lambda_s^{(2)}-\lambda_s^{(1)}) \psi(s),$$ with $\psi(s)=\frac{2}{z_s^{(1)}z_s^{(2)}}.$ 
The solution to this equation is
\begin{align}
H(s)=u(s)^{-1}\left(H(0)+\int_0^s(\lambda_r^{(2)}-\lambda_r^{(1)})u(r)\psi(r)dr\right),
\end{align}
with $u(s)=\exp \left(-\int_0^r \psi(s)ds\right).$
The estimate in Lemma \ref{l: h uniform close} is obtained for $H(0)=0$.
One can show (see \cite{viklund2014continuity}) that the following bound holds $u(s)^{-1}=\exp \left(\int_0^r \psi(s)ds\right) \leq \frac{(4s+y^2)^{1/2}}{y}$. Thus, when the two initial conditions have different real parts, one obtains and additional factor and the estimate reads
\begin{equation}\label{good estimate h uniform}
\begin{aligned}
|h^{(1)}_T(u_1)-h^{(2)}_T(u_2)| 
& 
\leq |\text{Re}(u_1)-\text{Re}(u_2)|\frac{I_{T,y}}{y}
\\
& +\epsilon \exp \left[\frac{1}{2}\left(\log A_1 \log A_2\right)^{1/2} +\log\log \frac{I_{T, y}}{y}\right],
\end{aligned}
\end{equation}
where 
\[
A_k=\frac{I_{T, y}|(h^{(k)}_{T})^{'}(u_)|}{y}, \quad k=1,2.
\]

\subsection{Proof of Theorem \ref{first main result}}

In this section we apply the results of the previous part in the case when the driving function is of the form $\sqrt{\kappa}B_t$. We are interested in how things change when $\kappa$ is changing. 

We start by discussing Assumptions \ref{assumption: osc} and \ref{assumption: derivative control}. It is easy to see that the first one is satisfied for sufficiently small $\delta$ since $osc(B_t,\delta)/\sqrt{2\delta\log(1/\delta)}\to1$. For all $\delta$ we use the following result.
\begin{proposition}[Theorem 3.2.4 in \cite{lawler2010random}]
\label{prop: Brownian oscilation}
Let $B_t$ be the standard Brownian motion on $[0,1]$. There is an absolute constant $c<\infty$ such that for all $0<\delta\le 1$ and $r>c$
\[
\P\brb{osc(B_t,\delta)\ge r\sqrt{\delta\log(1/\delta)}}\le c \delta^{(r/c)^2}.
\]
\end{proposition}
This means that if we take $r$ large enough, then we have a uniform bound on $osc$ with very high probability. Alternatively, for almost every $B_t$ there is (random) $r$ such that $osc\le r\sqrt{\delta\log(1/\delta)}$.
Throughout our analysis the driver is $\sqrt{\kappa}B_t$, with $\kappa \in (0, 8/3)$.  Thus, we can merge the constant $\kappa$ in the modulus of continuity of the driver $\sqrt{\kappa}B_t$ and estimate it directly with the biggest value. We will do the probabilistic version of this estimate in the next section.

Assumption \ref{assumption: derivative control} was established for SLE in \cite{friz2019regularity}.
\begin{proposition}[Corollary $4.2$ in \cite{friz2019regularity}] \label{Cor yizheng}
Let $\kappa_- > 0,$ and $\kappa_+ < 8/3$. Then there exist $\beta < 1$ and a random variable $C(\omega) < \infty$  such that almost surely
$$\sup_{(t, \kappa) \in [0,1] \times [\kappa_-, \kappa_+]}|\hat{f}'_t(iy)|
\leq C(\omega)y^{-\beta}$$
for all $ y \in [0,1]$. 
\end{proposition}

Let us consider two parameters $\kappa_1,\kappa_2 \in (0, 8/3)$ and two Loewner evolutions driven by $\sqrt{\kappa_1}B_t$ and $\sqrt{\kappa_2}B_t$. The corresponding maps and curves will be denoted by superscripts $(1)$ and $(2)$ correspondingly.

Throughout this section, the precise subpower function that we use is changing from line to line. Unless it might lead to a confusion, we do not track these changes in order to simplify notations.

Our goal is to estimate the supremum of $|\gamma^{(1)}(t)-\gamma^{(2)}(t)|$. 
By the triangle inequality
\begin{equation}\label{difftrac}
|\gamma^{(1)}(t)-\gamma^{(2)}(t)| \leq |\gamma^{(1)}(t)-\gamma^{n,(2)}(t)|+|\gamma^{n,(2)}(t)-\gamma^{(2)}(t)|
\end{equation}
where $\gamma^{n,(j)}$  is the trace obtained form interpolating with square root terms the driver $\sqrt{\kappa_j}B_t$.

In order to control the first term, we first fix an arbitrary interval $I = [t_k, t_{k+2}]$, with $ 0 \leq  k \leq n-2.$ We will estimate $|\gamma^{(1)}(s+t_k)-\gamma^{n, (2)}(r+t_k)|$ for all $r \in \left[\frac{1}{n}, \frac{2}{n} \right],$ and for  the specific point $s$ obtained in the Lemma \ref{SLEpointlemma}. 
Combining with the uniform continuity of $\gamma$ from Lemma \ref{unifcont}, we will have an estimate
for $|\gamma^{(1)}(r+t_k)-\gamma^{n, (2)}(r+t_k)|$ for all $r \in \left[\frac{1}{n}, \frac{2}{n} \right]$.  Redoing the same analysis on each interval in the time discretization, we obtain the desired estimate. The second term in the inequality \eqref{difftrac} is estimated by Theorem \ref{Ttran}.

Then, to be more precise, let $z=\gamma_k^{(1)}(s), w=\gamma_k^{n, (2)}(r)$, with $s$ and $r$, as before.

\begin{equation}
\label{inequalityTran}
\begin{aligned}
|\gamma^{(1)}(s+t_k)&-\gamma^{n, (2)}(r+t_k)|
\\ &\leq |\hat{f}^{(1)}_{t_{k}}(z)-\hat{f}^{(1)}_{t_{k}}(w)|+ |\hat{f}^{(1)}_{t_{k}}(w)-\hat{f}^{n, (2)}_{t_{k}}(w)  |.
\end{aligned}
\end{equation}

As in \cite{tran2015convergence}, we estimate the first term in \eqref{inequalityTran} using
$$|{\hat{f}^{(1)}}_{t_{k}}(z)-{\hat{f}^{(1)}}_{t_{k}}(w)| \leq (2\text{Im}z)|(\hat{f}^{(1)}_{t_{k}})^{'}(z)|\exp (4 d_{\mathbb{H}, hyp}(z,w)),$$

where $d_{\mathbb{H}, hyp}(z,w)=\operatorname{Arccosh}\left(1+\frac{|z-w|^2}{2\operatorname{Im}z \operatorname{Im}w} \right)$ is the hyperbolic distance in $\mathbb{H}.$ To estimate this we use the Proposition \ref{Cor yizheng}.

To estimate the second term in \eqref{inequalityTran} we use Lemma \ref{l: h uniform close}. For this we estimate the distance between the two driving terms: $\sqrt{\kappa_1}B_t$ and the square root interpolation of the $\sqrt{\kappa_2}B_t$

$$|\lambda^n_{\kappa_2}(t)-\sqrt{\kappa_1}B_t| \leq |\lambda^n_{\kappa_2}(t)-\sqrt{\kappa_2}B_t|+|\sqrt{\kappa_2}B_t-\sqrt{\kappa_1}B_t|\,. $$

Thus, we obtain combining the estimates with the ones in Subsection $3.2$ of \cite{tran2015convergence}, that

\begin{align*}
\epsilon := \sup_{t \in [0,1]}|\lambda^n_{\kappa_2}(t)-\sqrt{\kappa_1}B_t| &\leq \frac{\phi(n)}{\sqrt{n}}+|\sqrt{\kappa_1}-\sqrt{\kappa_2}|\sup_{t \in [0,1]}|B_t|\nonumber\\ &\leq\frac{\phi(n)}{\sqrt{n}}+c|\sqrt{\kappa_1}-\sqrt{\kappa_2}|\,.
\end{align*}

 Let $u_1 = x_1 + iy := w + \sqrt{\kappa_1}B_{t_k}$ and let $u_2=x_2+iy:=w+\lambda^n_{\kappa_2}(t_k)$. 
Then $\lambda^n_{\kappa_2}(t_k)$ is constructed such that $\lambda^n_{\kappa_2}(t_k)=\sqrt{\kappa_2}B_{t_k}$. Thus, we have that $|\text{Re}(u_1)-\text{Re}(u_2)| \leq |\sqrt{\kappa_1}-\sqrt{\kappa_2}|B_{t_k}$. By \eqref{good estimate h uniform}, we have that
\begin{equation}
\begin{aligned}
 |f^{(1)}_{t_k}(u_1)&-f^{n, (2)}_{t_k}(u_2)|\leq  |\sqrt{\kappa_1}-\sqrt{\kappa_2}|B_{t_k}\frac{I_{t_k, y}}{y}\\
 &+ \epsilon \exp \left[\frac{1}{2}\left( \log A_1\log A_2\right)^{1/2} +\log\log \frac{I_{t_k, y}}{y}\right],
 \end{aligned}
\end{equation}
where
\[
A_j=\frac{I_{t_k, y}|(f^{(j)}_{t_k})^{'}(u_j)|}{y}, \quad j=1,2
\]
with $$\epsilon \leq  \frac{2\phi(n)}{\sqrt{n}}+c|\sqrt{\kappa_1}-\sqrt{\kappa_2}|\,.$$
These estimates are used for points inside the boxes $A_{n,c, \phi}$. Thus, for $y=\text{Im}u_1=\text{Im}u_2=\text{Im}w \in [\frac{1}{\sqrt{n}{\phi(n)}}, \frac{2\sqrt{2}}{\sqrt{n}}]$
we have that
$$ \frac{I_{t_k, y}}{y} \leq 2\sqrt{2}\sqrt{n}\phi(n), $$
where $\phi(n)$ is some sub-power function of $n$\,.
Using that $\hat{f}'_{t_k}(w)=(f^{(1)}_{t_k})^{'}(u_1)$, we obtain from Proposition \ref{Cor yizheng} the estimate 
$$|(f^{(1)}_{t_{k}})^{'}(u_1)| \leq cy^{-\beta(\kappa_1)}\leq c\phi(n)^{\beta(\kappa_1)}\sqrt{n}^{\beta(\kappa_1)}\,.$$ 
and the general estimate
$$|(f^{n, (2)}_{t_k})^{'}(u_2)| \leq C(1/y+1)\leq 2C\phi(n)\sqrt{n}.$$ 
Note that the second estimate holds true for any conformal map of $\mathbb{H}$.
Combining these estimates, we obtain that
\[
\begin{aligned}
|&f^{(1)}_{t_k}(u_1)-f^{n, (2)}_{t_k}(u_2)|
\\ 
&\leq |\sqrt{\kappa_1}-\sqrt{\kappa_2}|B_{t_k}2\sqrt{2}\sqrt{n}\phi(n)
\\
&+ \frac{\phi(n)}{\sqrt{n}}\exp\left[\sqrt{\frac{1+\beta(\kappa_1)}{2}}\log(c\phi(n)\sqrt{n})+\log\log2\sqrt{2n}\phi(n) \right]
\\ 
&+ c|\sqrt{\kappa_1}-\sqrt{\kappa_2}|\exp\left[\sqrt{\frac{1+\beta(\kappa_1)}{2}}\log(c\phi(n)\sqrt{n})+\log\log2\sqrt{2n}\phi(n) \right] 
\\
 &\leq \frac{\phi(n)}{\sqrt{n}^{1-\sqrt{\frac{1+\beta(\kappa_1)}{2}}}}+ \Psi(|\sqrt{\kappa_1}-\sqrt{\kappa_2}|, n)+ \Phi( |\sqrt{\kappa_1}-\sqrt{\kappa_2}|, \kappa_1, n),
\end{aligned}
\]
where 
\begin{align*}
\Psi(|\sqrt{\kappa_1}-\sqrt{\kappa_2}|, n)=\Psi(n):= |\sqrt{\kappa_1}-\sqrt{\kappa_2}|\hat{c}2\sqrt{2}\sqrt{n}\phi(n)
\end{align*}
with $\hat{c} < \infty$ a.s. and
\begin{align*}
\Phi( &|\sqrt{\kappa_1}-\sqrt{\kappa_2}|, \kappa_1, n)=\Phi(n) \nonumber\\
 &:= c|\sqrt{\kappa_1}-\sqrt{\kappa_2}|\exp\left[\sqrt{\frac{1+\beta(\kappa_1)}{2}}\log(c\phi(n)\sqrt{n})+\log\log2\sqrt{2n}\phi(n) \right].
\end{align*}

Thus, using that $\hat{f}_{t_k}^{(1)}(w)-\hat{f}_{t_k}^{n, (2)}(w)=f_{t_k}^{(1)}(w+\sqrt{\kappa_1}B_t)-f_{t_k}^{n, (2)}(w+\lambda^n_{\kappa_2})$ and \eqref{inequalityTran} we obtain that 
\[
\begin{aligned}
|\gamma^{(1)}(s+t_k)&-\gamma^{n, (2)}(r+t_k)|
\\
&\leq \frac{\phi_1(n)}{\sqrt{n}^{1-\beta(\kappa_1)}} +\frac{\phi(n)}{\sqrt{n}^{1-\sqrt{\frac{1+\beta(\kappa_1)}{2}}}}+ \Psi(n)+ \Phi(n),
\end{aligned}
\]
for all $ r \in [\frac{1}{n}, \frac{2}{n}]$.
Using that $\sqrt{\frac{1+\beta}{2}} > \beta$, we obtain that 
\begin{align}\label{final}
|\gamma^{(1)}(s+t_k)&-\gamma^{n, (2)}(r+t_k)| \nonumber\\ &\leq \frac{\phi_2(n)}{\sqrt{n}^{1-\sqrt{\frac{1+\beta(\kappa_1)}{2}}}}+ \Psi(n)+ \Phi(n),
\end{align}
for all $r \in [t_{k+1}, t_{k+2}]$ and $0 \leq k\leq n-2$ and hence for all $r \in [0,1]$.

In order to estimate the second term in \eqref{difftrac}, i.e. $|\gamma^{(2)}(t)-\gamma^{n, (2)}(t)|$,  we use directly the result from Theorem \ref{Ttran}, i.e. we have that 
\[
|\gamma^{n,(2)}(t)-\gamma^{(2)}(t)| \leq \frac{\tilde{\phi}^{(2)}(n)}{n^{\frac{1}{2}\left(1-\sqrt{\frac{1+\beta(\kappa_2)}{2}}\right)}},
\]
where $\tilde{\phi}^{(2)}(n)$ is a subpower function that depends on the approximation of the driver $\sqrt{\kappa_2}B_t$. 

Thus, overall we have the following estimate that we control using the probabilistic estimates in the next section 
\begin{align*}
|\gamma^{(1)}(t)-\gamma^{(2)}(t)|\nonumber \leq \frac{\phi_2(n)}{\sqrt{n}^{1-\sqrt{\frac{1+\beta(\kappa_1)}{2}}}}+\Psi(n)+ \Phi(n)+ \frac{\tilde{\phi}^{(2)}(n)}{n^{\frac{1}{2}\left(1-\sqrt{\frac{1+\beta(\kappa_2)}{2}}\right)}}.
\end{align*}

\subsection{Probabilistic estimates}
We first consider the estimate in Proposition \ref{Cor yizheng} in order to obtain control on the derivative of $\hat{f}_t(z)$.
It follows from Proposition \ref{prop: Brownian oscilation} that there exists constants $c_1$ (depending on $\kappa$) and $c_2$ such that
\begin{align}
\mathbb{P}\left[ osc(\sqrt{\kappa}B_t, \frac{1}{m}) \geq c_1\sqrt{\frac{\log m}{m}} \right] \leq \frac{c_2}{n^2}\,.
\end{align}

Notice that in Theorem \ref{Ttran} the subpower function is $\phi(n)=\sqrt{\log(n)}$. Then as in \cite{tran2015convergence}, by going through the
proof, one sees that the subpower functions are changed by adding, multiplying and
exponentiating constants. Hence the if we merge the dependence on $\kappa$ in the initial subpower function, i.e. we start with $\sqrt{\kappa \log n}$, then we end up with 
$c(\sqrt{\kappa \log n})^{c'}$ for some constants $c$ and $c'\,.$
Using \eqref{final} we obtain that
\[
\begin{aligned}
&\mathbb P \left[\|\gamma^{(1)}-\gamma^{n, (2)}\|_{[0,1], \infty} \leq \frac{c_6(\kappa\log m)^{c_7}}{\sqrt{m}^{1-\sqrt{\frac{1+\beta(\kappa_1)}{2}}}} + \Psi (m)+ \Phi(m)  \hspace{2mm} \text{for all} \hspace{2mm} m\geq n \right] \nonumber\\&\geq 1-\left(\frac{c_2(\kappa_1)}{n^2} +\frac{c_3(\kappa_1)}{n^{c_4(\kappa_1)/2}}\right).
\end{aligned}
\]
Next, by \cite{tran2015convergence}, there are $c_4$ and $c_5$ depend on $\kappa_2$ such that 
\begin{align*}
\mathbb{P}\left[\|\gamma^{(2)}(t)-\gamma^{m, (2)}(t)\|_{[0,1], \infty} \leq \frac{c_1(\kappa_2\log(m))^{c_2}}{\sqrt{m}^{1-\sqrt{\frac{1+\beta(\kappa_2)}{2}}}} \hspace{1mm} \text{for all} \hspace{2mm} m\geq n   \right] \geq 1-\frac{c_4}{n^{c_5}},
\end{align*}
The previous analysis performed for the two values $\kappa_1$ and $\kappa_2$ can be extended for sequences $\kappa_j \to \kappa$. 

First, we apply the previous Lemmas \ref{SLEpointlemma} and \ref{LTran2}  for sequences $\kappa_j \to \kappa$. Next, we use the almost sure estimate from Proposition \ref{Cor yizheng} for the sequence $\kappa_j \to \kappa$ by making use that the constant $C(\omega)$ in $$\sup_{(t, \kappa) \in [0,1] \times [\kappa_-, \kappa_+]}|\hat{f}'_t(iy)|
\leq C(\omega)y^{-\beta}$$ does not depend on the sequence $\kappa_j \to \kappa$.

Continuing the analysis, the sizes of the boxes $A_{n, c, \phi}$ depend on $\kappa_j$ via the dependence of the subpower function that we choose, on $\beta=\beta(\kappa_j)$ and on $\phi$ (that depends also on $\kappa_j$, since the driver is $\sqrt{\kappa_j}B_t$).  However, since the constant $c=2\sqrt{2}$ is fixed, the upper level of the boxes remains the same as we consider $\kappa_j \to \kappa$, only their width and lower level changes. 

We consider $\kappa_j \to \kappa$ by choosing for each $j$ the largest box that contains both points $z$ and $w$ in order to estimate the hyperbolic distance between them, i.e. we make use of the fact that the upper height of the boxes coincides and we work on $A_{n, 2\sqrt{2}, \xi(\kappa_j)}$ with $\xi(\kappa_j)=\max (\psi(\kappa_j), \tilde{\psi}(\kappa))$. This is a dynamical version (as we vary the index $j$) of the analysis in \cite{tran2015convergence} that is performed for fixed $\kappa$.  For each fixed $j$, the estimates work in the same manner.

In order to assure that $\Psi( |\sqrt{\kappa}-\sqrt{\kappa_j}|, n)$ and $\Phi( |\sqrt{\kappa}-\sqrt{\kappa_j}|, \kappa, n)$ converge to zero as $j \to \infty$,  we choose $n=n(\kappa_j)$ such that as $j \to \infty$
$$\hat{c} |\sqrt{\kappa}-\sqrt{\kappa_j}|2\sqrt{2}\sqrt{n}\phi(n) \to 0$$ and 
$$  c|\sqrt{\kappa}-\sqrt{\kappa_j}|\exp\left[\sqrt{\frac{1+\beta(\kappa)}{2}}\log(c\phi(n)\sqrt{n})+\log\log2\sqrt{2n}\phi(n) \right] \to 0. $$

Combining the previous estimates and using a union bound, we obtain the result. 

For the second part of the result, the continuity in $\kappa$ for $\kappa \in (0, 8/3)$ of the curves generated by the algorithm is obtained by estimating

\[
\begin{aligned}
|\gamma^{n, (1)}(t)-\gamma^{n, (2)}(t)|&\leq|\gamma^{n, (2)}(t)-\gamma^{(2)}(t)| 
\\
&+|\gamma^{(2)}(t)-\gamma^{(1)}(t)| +| \gamma^{(1)}(t)-\gamma^{n, (1)}(t)|.
\end{aligned}
\]

The first and the last term can be directly estimated using Theorem  \ref{Ttran}, since these are terms that compare the $SLE_{\kappa_1}$ and $SLE_{\kappa_2}$  traces with the corresponding approximated traces. The middle term is estimated using the analysis performed in the proof so far, and the conclusion follows.

\begin{remark}
The algorithm uses estimates on the derivative of the conformal maps. We remark that the derivative of the composition of the conformal maps obtained when solving Loewner equation on each element of the partition of the time interval $[t_{k}, t_{k+1}]$ (where $t_k=\frac{k}{n}$, $0\leq k\leq n$) with $c\sqrt{t}+d$ with $c$, $d \in \mathbb R $, is not easy to estimate directly. That is why we used in our proof the estimate on the derivative of the Loewner map $$\sup_{(t, \kappa) \in [0,1] \times [\kappa_-, \kappa_+]}|\hat{f}'_t(iy)|
\leq C(\omega)y^{-\beta}$$ from Proposition \ref{Cor yizheng} with $\beta(\kappa) <1$, $\forall \kappa \neq 8.$ 
\end{remark}

\section{Proof of the second main result}

\subsection{Rough Path Theory overview}
First, in this subsection we give an overview of Rough Path Theory following \cite{lyons1998differential}, that we refer the reader to for more details.

For $T>0$ a real number and $V$ a finite dimensional vector space, we let $X_{[s,t]}$ denote the restriction of the continuous function $X:[0,T] \to V$ to the compact interval $[s,t]$. Next, we introduce the notion of $p$-variation. 
\begin{definition}
Let $V$ be a finite dimensional real vector space with dimension $d$ and basis vectors $e_1,\ldots, e_d\,.$. The $p$-variation of a path $X: [0,T] \to V $ is defined by 
\begin{align*}
||X_{[0,T]}||_{p-var} \; := \; \sup_{\mathcal{D}=(t_0, t_1,\ldots,t_n) \subset [0,T]}\left( \sum\limits_{i=0}^{n-1}d(X_{t_i}, X_{t_{i+1}})^p \right)^{\frac{1}{p}}\,,
\end{align*}
where the supremum is taken over all finite partitions of the interval $[0,T]\,.$
\end{definition}
 Throughout the next sections we use the notation $X_{s,t}= X_t-X_s\,.$
Let us further define $\Delta_T = \{ (s,t) \in [0, T] \times [0,T] |0 \leq s\leq t \leq T\}.$
We introduce next the fundamental notion of control.
\begin{definition}
 A control on $[0, T]$ is a non-negative continuous function $$ \omega : \Delta_T \to [0, \infty)$$ for which
\begin{align*}
\omega(s,t)+\omega(t,u) \leq \omega(s,u),
\end{align*}
for all $0 \leq s \leq t \leq u \leq T, $ and $\omega (t,t) = 0,$  for all $t \in [0,T]\,.$
\end{definition}
Furthermore, we introduce the following.
\begin{definition}
Let $T((V)):= \left\{\boldsymbol{a} = (a_{0}, \,a_{1},\, \cdots\,) : a_{n}\in V^{\otimes n}\,\,\,\forall n\geq 0\right\}$ denote the set of formal series of tensors of $V\,.$
\end{definition}
\begin{definition}
The tensor algebra $T(V);= \bigoplus_{k\geq 0}V^{\otimes_k}$ is the infinite sum of all tensor products of $V\,.$
\end{definition}

Let $e_1, e_2, \ldots, e_d $ be a basis for $V\,.$ The space $V^{\otimes k}$ is a $d^k$ dimensional vector space with basis elements of the form $(e_{i_1} \otimes e_{i_2} \ldots \otimes e_{i_k})_{(i_1,\ldots,i_k) \in \{ 1,\ldots, d \}^k} \,.$ We store the indices $(i_1, \ldots , i_k) \in \{1,2, \ldots d \} ^{k}$ in a multi-index $I$ and let $ e_{I}=e_{i_1} \otimes e_{i_2} \otimes \ldots e_{i_k}\,.$
The metric $||\cdot||$ on $T((V))$ is the projective norm defined for $$x= \sum_{|I|=k}\lambda_I e_I \in V^{\otimes k}$$ via $$||x||=\sum_{|I|=k}|\lambda_I|.$$  Thus, the bound $||X^i_{s,t}|| \leq \frac{{w(s,t)}^{i/p}}{\beta(\frac{i}{p})!}, \hspace{2mm} \forall i \geq 1,\hspace{2mm} \forall (s,t) \in \Delta_T,$ gives control on the sum of $i$-iterated integrals.
We collect all the iterated integrals in the following way.
We consider for $$X: \Delta_T \to T((\mathbb{R}))$$ the collection of iterated integrals as $$(s,t) \to \mathbf{X}_{s,t}= (1, X_{s,t}^1, \ldots, X_{s,t}^{[p]}, \ldots, X_{s,t}^m, \ldots ) \in T((V)).$$  We call the collections of iterated integrals the signature of the path $X$. 

We now define the notion of $\textit{multiplicative functional}$.
\begin{definition}
Let $n \geq 1$ be an integer and let $X : \Delta_T \to T^{(n)}(V)$ be a continuous map. Denote by $X_{s,t}$ the image of the interval $(s,t)$ by $X\,,$ and write
$$X_{s,t}=(X_{s,t}^{0}, \ldots X_{s,t}^{n}) \in \mathbb{R}\oplus V\oplus V^{\otimes 2} \ldots \oplus V^{\otimes n}\,.$$
The function $X$ is called multiplicative functional of degree $n$ in $V$ if $X_{s,t}^0=1$ and for all $(s,t) \in \Delta_t$ we have
$$X_{s,u} \otimes X_{u,t}= X_{s,t} \hspace{2mm} \forall s,u,t \in [0,T]\,.$$
\end{definition}
Throughout our analysis, we will use the notion of $p$-rough path that we define in the following. 
\begin{definition}\label{def1}
A $p$-rough path of degree $n$ is a map $X : \Delta_T \to \tilde{T}^{(n)}(V)$ which satisfies Chen's identity $X_{s,t}\otimes X_{t,u}=X_{s,u}$ and the following 'level dependent' analytic bound
$$||X_{s,t}^{i}|| \leq \frac{w(s,t)^{\frac{i}{p}}}{\beta_p (\frac{i}{p})!} \,,$$
\end{definition}
\noindent
where $y!=\Gamma(y+1)$ whenever $y$ is a positive real number and $\beta_p\,,$ is a positive constant. 

Furthermore, we introduce a metric on $\Lambda_p(V)$ which transform the space $\Lambda_p(V)$ in a complete metric space. For $X, Y \in \Lambda_p(V)$ we define 
$$d_p(X,Y)=\max_{1\leq i \leq [p]}\sup_{\cal{D} \subset [0,T]} \left( \sum\limits_{\cal{D}}||X_{t_i, t_{i+1}}^i-Y_{t_i, t_{i+1}}^i ||^{\frac{p}{i}}\right)^\frac{i}{p}. $$

Related to this notion is a notion of convergence that is the \textit{convergence in the p-variation topology}\,. Formally, this is defined in terms of converging sequences. 
\begin{definition}\label{pvarconv}
A sequence $(X(n))_{n \geq 1} \in \Lambda_p(V)$ is said to converge to $X \in \Lambda_p(V)$ in $p$-variation topology if there exists a $p$-control $w$ of $X$ and $X(n)$ for all $n \geq 1\,,$ and a sequence $(a(n))_{n \geq 1}$ of positive reals such that $\lim_{n \to \infty}a(n)=0$ and 
$$||X(n)_{s,t}^{i}-X_{s,t}^{i}|| \leq a(n)w(s,t)^{\frac{i}{p}}\,, $$ for all $(s,t) \in \Delta_T$ and $ 1 \leq i \leq [p]\,.$ 
\end{definition}

We are now ready to define the notion of a geometric rough path.
\begin{definition}
A geometric $p$-rough path is a $p$-rough path that can be expressed as a limit of $1$-rough paths in the $p$-variation metric. 
\end{definition}
The space of geometric $p$-rough paths in $V$ is denoted by $G \Omega_p(V)\,.$

In order to state our second main result, we need to introduce the notion of $Lip (\gamma) $ function (that we define more generally in order to follow the exposure in \cite{lyons2007differential}).
\begin{definition}
Let $V$ and $W$ be two Banach spaces. Let $k \geq 0$ be an integer. Let $\gamma \in (k, k+1]$ be a real number. Let $F$ be a closed subset of $V$. Let $f : F \to W$ be a function. For each integer $j = 1,\ldots, k$ let $f^j: F \to \mathbf{L}(V^{\otimes j}, W)$ be a function which takes its values in the space of $j$-linear mappings from $V$ to $W.$ The collection ($f=f^0, f^1, \ldots, f^k$) is an element of $Lip(\gamma, F)$ if the following condition holds.

There exists a constant $M$ such that, for each $j=0, \ldots, k, $
\begin{equation*}
\sup_{x\in F}|f^j(x)| \leq M
\end{equation*}
and there exists a function $R_j : V \times V \to \mathbf{L}(V^{\otimes j}, W) $ such that, for each $x, y \in F$ and each $v \in V^{\otimes j}, $ we have 
\begin{equation*}
f^j(y)(v)=\sum\limits_{l=0}^{k-j}\frac{1}{l!}f^{j+l}(x)(v\otimes(y-x)^{\otimes l})+R_j(x,y)(v)\,,
\end{equation*}
and
\begin{equation*}
|R_j(x,y)| \leq M |x-y|^{\gamma-j}\,.
\end{equation*}
The smallest $M$ for which the inequalities hold for all $j$ is called the $Lip (\gamma, F)$-norm of $f$\,.
\end{definition}

Following \cite{lyons2007differential}, when $V$ is finite dimensional, we obtain that there exist for all closed $F$ a continuous extension operator $Lip(\gamma, F) \to Lip (\gamma, V)$. Thus, in this manner we obtain $Lip(\gamma, V)=Lip (\gamma)$, i.e. bounded continuous functions on $V$ which are $k$-times continuously differentiable with bounded derivatives on $V$ and whose $k$-th differential is H\"older continuous with parameter $\gamma-k$.

\subsection{ Universal Limit Theorem and the backward Loewner differential equation }

In the next sections, we work with the backward Loewner differential equation

\begin{equation}\label{6}
\partial_{t}h(t,z)=\frac{-2}{h(t,z)-\sqrt{\kappa}B_{t}}\,, \hspace{10mm} h(0,z)=z, z \in \mathbb{H},
\end{equation}
where $0 \neq \kappa \in \mathbb{R}_+$ and $B_t$ a standard one-dimensional Brownian motion.
By performing the identification $Z_t=h_t(z)-\sqrt{\kappa}B_t$, we obtain the following dynamics in $\mathbb{H}$ that we consider throughout this section
 $$dZ_t=\frac{-2}{Z_t}dt-\sqrt{\kappa}dB_t, \hspace{5mm}Z_0=z_0 \in \mathbb H.$$

We consider the backward Loewner differential equation started from $Z_0 \in \mathbb{H}$ with $|Z_0|=\delta$, for any $\delta>0$. Furthermore, one can write the backward Loewner differential equation as $dZ_t=V(Z_t)dX_t,$
with $V(Z)=(V_1(Z), V_2(Z)),$ where $V_1(Z)=\frac{-2}{Z}\frac{d}{dz}$ and $V_2(Z)=\sqrt{\kappa}\frac{d}{dx}$ are the two vector fields of the equation. Moreover, the equation is driven by the two-dimensional path $X_t=(t, B_t)$.

\begin{remark}[Geometric Rough Path lift of $X_t=(t, B_t)$]
When discussing continuity properties of the solution to the backward Loewner differential equation with respect to the parameter $\kappa$, we need that the pair $X_t=(t, B_t)$ to be a geometric rough path. Thus, we need to consider a different lift from the It\^o one. Since $t$ is of bounded variation, the pair $X_t$ is a Young pairing. Since $B_t$ is one-dimensional Brownian motion, then there is  a canonical lift to a geometric rough path (for higher dimensions, it is shown  in \cite{sipilainen1993pathwise} that the Stratonovich lift of the Brownian motion is a geometric rough path). We use this lift to see the pair $X_t=(t, B_t)$ as a geometric $p$-rough path for $p>2$. For further details, see Section $9.4$ in \cite{friz2010multidimensional}.
\end{remark}
We remark also the following.
\begin{remark}\label{Ito=stratono}
In the case of the backward Loewner differential equation driven by $\sqrt{\kappa}B_t$ the It\^o lift or the Stratonovich lift of the the iterated integrals produce the same solution. Indeed, when considering the It\^o-Stratonovich correction for a time-homogeneous diffusion $$dZ_t=\mu(Z_t)dt+\sigma(Z_t)dB_t,$$ we have
$$\int_0^T\sigma(Z_t)\circ dB_t=\frac{1}{2}\int_{0}^T\frac{d \sigma(Z_t)}{dx}\sigma(Z_t)dt+\int_0^T\sigma(Z_t)dB_t.$$
In our case, since $dZ_t=\frac{-2}{Z_t}dt-\sqrt{\kappa}dB_t$, we have that $\sigma(Z_t)$ is a constant. Thus, when studying this equation we obtain $$\int_0^T\sigma(Z_t)\circ dB_t=\int_0^T\sigma(Z_t)dB_t.$$\\
\end{remark}
We further define the notion of solution to a Rough Differential Equation. 
\begin{definition}\label{defsolRDETerry}
Let $f : W \to L(V, W)$ be a $Lip(\gamma-1)$ function and let us consider $$X \in G\Omega_p(V)\hspace{2mm} \text{and} \hspace{2mm} \zeta \in W.$$
Set $f_{\zeta}(\cdot)=f(\cdot+ \zeta)$. Define $h: V\oplus W \to End(V \oplus W)$ via
$$\begin{bmatrix}
Id_V & 0\\
f_{\zeta}(y) & 0\\
\end{bmatrix}.$$
We call $Z \in G \Omega_p(V \oplus W)$ a solution to the differential equation $dY_t=f(Y_t)dX_t$, $Y_0=\zeta>0$ if the following conditions hold
\begin{itemize}
\item $Z=\int h(Z)dZ$,
\item $\Pi_V(Z)=X$, where $\Pi_V(\cdot)$ is the projection map to the first component.
\end{itemize}
\end{definition}
The main result that we use in our proof is the following theorem.
\begin{theorem}[Universal Limit Theorem, Theorem $5.3$ in \cite{lyons2007differential}]\label{UniversalLimit}
Let $p \geq 1$ and let $\gamma > p$ be real numbers. Let $f :W \to L(V, W)$ be a $Lip(\gamma)$ function. For all $X \in G\Omega_p(V)$ and all $\zeta \in W$, the equation 
$$dY_t=f(Y_t)dX_t, \hspace{2mm} Y_0=\zeta$$
admits a unique solution $Z=(X,Y) \in G\Omega_p(V \oplus W)$ in the sense of the Definition  \ref{defsolRDETerry}. The solution depends continuously on $X$ and $\zeta$ and the mapping $$I_f: G\Omega_p(V) \times W \to G\Omega_P(W)$$ which sends $(X, \zeta)$ to $Y$ is the unique extension of the It\^o map which is continuous in the $p$-variation topology. 
\end{theorem}

In the remaining part of the paper we prove the following theorem.

\begin{theorem}\label{continuityflow}
For $\delta>0$, the backward Loewner differential equation driven by $\sqrt{\kappa}B_t$ with $\kappa \in \mathbb{R}_+$, $\kappa \neq 0$, started from $z_0$ with $|z_0|=\delta>0$ is a well defined Rough Differential Equation that has a  unique solution.
This unique solution is a.s. continuous with respect to the starting point $z_0 \in \mathbb{H}$ and $\sqrt{\kappa}B_t$ in the $p$-variation topology, for $p \in (2,3]$. 
\end{theorem}

\subsection{Proof of Theorem \ref{continuityflow}}

We first prove the following lemma.

\begin{lemma}\label{lemmacontinuitykappa}
 For  $\kappa > 0$, let us consider $\kappa_n \to \kappa$, as $n \to \infty$. Then, $$(t,\sqrt{\kappa_n}B_t) \to (t,\sqrt{\kappa}B_t)$$ in the $p$-variation topology as $\kappa_n \to \kappa,$ for $p \in (2,3]$. 
\end{lemma}

\begin{proof}[Proof of Lemma \ref{lemmacontinuitykappa}]
Let is consider  the sequence $\sqrt{\kappa_n} \to \sqrt{\kappa}$ as $n \to \infty$, for $\kappa >0$. Since in the first component there are no changes, we focus directly on the second component of the paths $(t, \sqrt{\kappa_n}B_t)$ and $(t, \sqrt{\kappa}B_t)$. Without loss of generality we can choose an increasing sequence. Using the control $\omega(s,t)=\sqrt{\kappa}(t-s)$ and the sequence $a^{\kappa}_n=\frac{\sqrt{\kappa}-\sqrt{\kappa_n}}{\sqrt{\kappa}}$ in the Definition \ref{pvarconv} we obtain the convergence of in the $p$-variation topology for  $p \in (2,3]$ of the paths $(t, \sqrt{\kappa}_nB_t)$ and $(t, \sqrt{\kappa}B_t)$. Indeed,  we have that  $\omega(s,t)=\sqrt{\kappa}(t-s)$  is  a control for both $\sqrt{\kappa}B_{s,t}$ as well as $ \sqrt{\kappa_n}B_{s,t}$ and $|\sqrt{\kappa}B_{s,t} -\sqrt{\kappa_n}B_{s,t}|\leq |\sqrt{\kappa}B_{s,t}||1-\frac{\sqrt{\kappa_n}}{\sqrt{\kappa}}|.$

Thus, for any $0 \neq \kappa \in \mathbb{R}_+$ we have $\lim_{n \to \infty}a^{\kappa}_n \to 0$. Then, the bound in the Definition \ref{pvarconv}, holds for all pairs $s, t \in \Delta_T$. It can be directly checked that the same convergence result holds with the choice $w(s, t):=\sqrt{\kappa}(t-s)$ for higher levels $1< i \leq [p]$, and we obtain the desired result. 
\end{proof}

We are now ready to prove the second result of this paper.

\begin{proof}[Proof of Theorem \ref{continuityflow}]
We consider the geometric $p$-rough path for lift for $X_t=(t, B_t)$.
In order to prove that the backward Loewner differential equation started from $z_0 \in \mathbb{H}$ with $|Z_0|=\delta>0$ is a Rough Differential Equation with a unique solution, we show that the vector fields $V_1(Z)$ and $V_2(Z)$ are indeed $Lip(\gamma)$ vector fields for $\gamma>2$. The problematic vector field is $V_1(z)$ since the second one is clearly $Lip (\gamma)$ for $\gamma \geq 4$.

In order to show that indeed the first vector field is $Lip(\gamma)$ for $\gamma>3$ we use that
$$\frac{d}{dz}\frac{1}{z}=\frac{-1}{z^2},$$ $$\frac{d}{dz}\frac{-1}{z^2}=\frac{2}{z^3},$$ $$\frac{d}{dz}\frac{2}{z^3}=\frac{-6}{z^4}$$ 
and $$\frac{d}{dz}\frac{-6}{z^4}=\frac{24}{z^5}.$$ 
Thus, the $Lip(4)$ norm of the vector field $1/z$ is bounded for $|z|>\delta>0$. 
Note that, in general, $\frac{d^n}{dz}\frac{1}{z}=\frac{c(n)}{z^{n+1}}$, where the function $c(n)=(-1)^nn!$. Thus, the $Lip(\gamma)$ norm of the vector fields is bounded for any finite $\gamma$. In our analysis, we restrict to $p \in (2,3]$ and thus checking $\gamma >3$ is enough. 

In addition, we observe that for the imaginary part of the backward Loewner differential equation $Y_t$, we have
$$dY_t=\frac{2Y_t}{X_t^2+Y_t^2}dt\,.$$
 In particular, $Y_t\geq Y_0$, for all $t \geq 0$. Thus, we have that $Y_t>0$ for all $t \in [0, T]$. 
Since the imaginary part $Y_t$ increases the vector field $V_1(Z)$ remains bounded for all times $t \in [0,T].$ Using the Stratonovich lift of the pair $(t, B_t)$ and the bounds on $Lip(3)$ norms of the vector fields, we obtain that indeed the backward Loewner differential equation started from $z_0$ with $|z_0|=\delta>0$ is a Rough Differential equation. In particular, since the vector fields of the backward Loewner differential equation are $Lip(\gamma)$ for $\gamma \geq 3$ we obtain applying Theorem \ref{UniversalLimit} that the solution of the backward Loewner differential equation driven by $\sqrt{\kappa}B_t$, started from $z_0 \in \mathbb{H}$ with $|z_0|=\delta>0$, exists and is unique.

Let us consider the paths $X^{\kappa_n}_t=(t, \sqrt{\kappa_n}B_t)$  and $X^{\kappa}_t=(t, \sqrt{\kappa}B_t)$. Using Lemma \ref{lemmacontinuitykappa}, for $\kappa_n \to \kappa$, as $n \to \infty$, we obtain that $X^{\kappa_n}_t \to X^{\kappa}_t$ in the $p$-variation topology, for $p \in (2,3].$ 
Next, by applying Theorem \ref{UniversalLimit} we obtain that the solution of the backward Loewner differential equation started from $\delta>0$, for any $\delta>0$, is continuous in both staring point and $X_t=(t, \sqrt{\kappa}B_t)$ in the $p$-variation topology, for $p \in (2,3]$, and the conclusion follows.  
\end{proof}

\bibliographystyle{plain}
\bibliography{literature.bib}


\end{document}